\documentclass[10pt, reqno]{amsart}
\usepackage{graphicx, amssymb, amsmath, amsthm}
\numberwithin{equation}{section}

\let\Re=\undefined\DeclareMathOperator*{\Re}{Re}

\newcommand{\R}{\mathbb{R}}
\newcommand{\C}{\mathbb{C}}

\newcommand{\la}{\lambda}

\newtheorem{theorem}{Theorem}[section]

\newtheorem{lemma}[theorem]{Lemma}

\newtheorem{proposition}[theorem]{Proposition}

\theoremstyle{definition}
\newtheorem{definition}[theorem]{Definition}
\newtheorem{remark}[theorem]{Remark}

\theoremstyle{remark}

\makeatletter
\newcommand{\Extend}[5]{\ext@arrow0099{\arrowfill@#1#2#3}{#4}{#5}}
\makeatother

\begin{document}
\title[Scattering for NLS with inverse-square potential]{Scattering theory for nonlinear Schr\"odinger equations with inverse-square potential}
\author{Junyong Zhang}
\address{Department of Mathematics, Beijing Institute of Technology, Beijing 100081 China,
and Department of Mathematics, Australian National University,
Canberra ACT 0200, Australia} \email{zhang\_junyong@bit.edu.cn}
\author{Jiqiang Zheng}
\address{Universit\'e Nice Sophia-Antipolis, 06108 Nice Cedex 02, France, and Institut Universitaire de France}
\email{zhengjiqiang@gmail.com} \maketitle

\begin{abstract}
We study the long-time behavior of solutions to nonlinear
Schr\"o-dinger equations with some critical rough potential of
$a|x|^{-2}$ type. The new ingredients are the interaction
Morawetz-type inequalities and Sobolev norm property associated with
$P_a=-\Delta+a|x|^{-2}$. We use such properties to obtain the
scattering theory for the defocusing energy-subcritical nonlinear
Schr\"odinger equation with inverse square potential in energy space
$H^1(\R^n)$.
\end{abstract}

 \maketitle

\begin{center}
 \begin{minipage}{100mm}
 \end{minipage}
 \end{center}





\section{Introduction}

This paper is devoted to the scattering theory for the nonlinear
defocusing  Schr\"odinger equation
\begin{equation}\label{Eq3}
\begin{cases}
i\partial_{t} u-P_a u=|u|^{p-1}u\quad (t,x)\in\R\times\R^n\\
u|_{t=0}=u_0\in H^1(\R^n)
\end{cases}
\end{equation} where $u:\R_t\times\R_x^n\to \C$ and
$P_a=-\Delta+a|x|^{-2}$ with $a>-\la_n:=-(n-2)^2/4$ and $n\geq3$.
The elliptic operator $P_a$ is the self-adjoint extension of
$-\Delta+a|x|^{-2}$. It is well-known that in the range
$-\la_n<a<1-\la_n$, the extension is not unique; see
\cite{KSWW,Tit}. In this case, we do make a choice among the
possible extensions, such as Friedrichs extension \cite{KSWW,PSS}.

\vspace{0.2cm}

The scale-covariance elliptic operator $P_a=-\Delta+a|x|^{-2}$
appearing in \eqref{Eq3} plays a key role  in many problems of
physics and geometry. The heat and Schr\"odinger flows for the
elliptic operator $-\Delta+a|x|^{-2}$ have been studied in the
theory of combustion (see \cite{LZ}), and in quantum mechanics (see
\cite{KSWW}). The mathematical interest in these equations with
$a|x|^{-2}$ however comes mainly from the fact that the potential
term is homogeneous of degree $-2$ and therefore scales exactly
the same as the Laplacian. There is extensive literature on
properties of the Schr\"odinger semigroup of operators
$e^{it\mathrm{H}}$ generated by $\mathrm{H}=-\Delta+V(x)$, where a
potential $V(x)$ is less singular than the inverse square potential
at the origin, for instance,  when it belongs to the Kato class; see
\cite{DFVV,Schlag,SSS1,SSS2}. The inverse square potential
$a|x|^{-2}$ does not belong to the Kato class and it is well known
that such singular potential belongs to a borderline case, where
both the strong maximum principle and Gaussian bound of the heat
kernel for $P_a$ fail to hold when $a$ is negative. Because of this,
it brings some difficulties to study the heat and dispersive
equations with the inverse square potential; see \cite{BPSS,LZ}.
Fortunately, the Strichartz estimates, an essential tool for
studying the behavior of solutions to nonlinear Schr\"odinger
equations and wave equations, have been developed by
Burq-Planchon-Stalker-Tahvildar-Zadeh \cite{BPSS,BPST2}. In the study of Strichartz estimates for
the propagators $e^{it(\Delta+V)}$, the decay $V(x)\sim |x|^{-2}$ is borderline in order to guarantee validity of Strichartz
estimate; see Goldberg-Vega-Visciglia \cite{GVV}. And also it is known that for any potential $V(x)\sim |x|^{-2-\epsilon}$,
the Strichartz estimates are satisfied, without any further assumption on the monotonicity or regularity of $V(x)$; see Rodnianski-Schlag \cite{RS}. Moreover 
recently the Hardy type potentials have been further studied in Fanelli-Felli-Fontelos-Primo \cite{FFFP}. It is
well-known that inverse square potential is in some sense critical
for the spectral theory. This is closely related to the fact that
the angular momentum barrier $k(k+1)/|x|^2$ is exactly same type as
the inverse square potential. As a consequence, the authors
\cite{MZZ,ZZ} showed some more Strichartz estimates and restriction
estimates for wave equation with inverse square potential by
assuming additional angular regularity. In this paper, we study the
scattering theory of nonlinear Schr\"odinger equation \eqref{Eq3}
with the critical decay inverse square potential.\vspace{0.1cm}

The scattering theory of the nonlinear Schr\"odinger equation with no potential, that is
$a=0$,
has been intensively studied in
\cite{Bour98,Bo99a,Cav,CaW,CKSTT08,GV79,GV85}. For the
energy-subcritical case: $p\in(1+\tfrac4n,1+\tfrac4{n-2})$ when
$n\geq3$, and $p\in(1+\tfrac4n,+\infty)$ when $n\in\{1,2\}$, one can obtain the global well-posedness
for \eqref{Eq3} with $a=0$ by using the mass and energy
conservation due to the lifespan of the local solution depending only on the $H^1$-norm of
the initial data. In \cite{GV85}, Ginibre-Velo established the
scattering theory in the energy space $H^1(\R^n)$ by using the
classical Morawetz estimate for low spatial and almost finite
propagator speed for high spatial.  The dispersive estimate is
an essential tool in their argument. However in our setting, in
particular when $a$ is negative, even though we have the Strichartz
estimates, we do not know whether the dispersive estimate holds or
not. Later, Tao-Visan-Zhang \cite{TVZ} gave a simplified proof for
the result in \cite{GV85} by making use of the following interaction
Morawetz estimate
\begin{equation}\label{mor}
\big\||\nabla|^{\frac{3-n}4}u\big\|_{L_t^4(I;L_x^4(\R^n))}^2\leq
C\|u_0\|_{L^2}\sup_{t\in I}\|u(t)\|_{\dot H^{\frac12}},\quad n\geq3.
\end{equation} For this estimate, Visciglia \cite{V}  gave an
alternative application of the interaction Morawetz estimate.

To prove the scattering theory, we follow Tao-Visan-Zhang's
argument. Thus one of our main task is to establish an interaction
Morawetz estimate for the nonlinear Schr\"odinger equation
\eqref{Eq3} with inverse square potential. To this end, we have to
treat an error term caused by the potential. Though we cannot show
the positivity of this error term, we can control the error term by
the quantity in the right hand side of \eqref{mor}. The method to control the error term is, as with the classical
Morawetz inequality, a `multiplier' argument based on the first order differential operator $A=\frac12(\partial_r-\partial_r^*)$.
The two key points are the positivity of the commutator $[A,\Delta-a|x|^{-2}]$ and the homogeneity of  the potential $a|x|^{-2}$, which is the same as
Laplacian's scaling. Thus we finally obtain an
analogue of the interaction Morawetz-type estimate. We are known that
the Leibniz rule plays a role in proving the well-posedness.
However, we do not know whether the Leibniz rule associated with the
operator $P_a$ holds or not. Instead, we show the equivalence of the
Sobolev norms based on the operator $P_a$ and the standard Sobolev
norms based on the Laplacian by using results on the boundedness of
Riesz transform Hassell-Lin \cite{HL} and heat kernel estimate
\cite{LS,MS}. Though the Sobolev norms equivalence result partially
implies Sobolev algebra property associated with $P_a$, it is enough
for considering the scattering theory in energy space $H^1(\R^n)$ to
obtain our main result. \vspace{0.2cm}

The main purpose of this paper is to prove the following result.
\begin{theorem}\label{thm} Let $n\geq3$ and let $p\in\big(1+\tfrac4n,1+\tfrac4{n-2}\big)$. Assume that $a>-\frac{4p}{(p+1)^2}\lambda_n$ and $u_0\in H^1(\R^n)$.
Then the solution $u$ to \eqref{Eq3} is global. Moreover, the solution
$u$ scatters if $a\geq\frac4{(p+1)^2}-\lambda_n$ for $n\geq4$, and $a\geq0$ for $n=3$.

Here the solution $u$ to \eqref{Eq3} scatters means that there exists  a unique $u_\pm\in H^1(\R^n)$ such that
$$\lim_{t\to\pm\infty}\|u(t)-e^{itP_a}u_\pm\|_{H^1_x}=0.$$
\end{theorem}

\begin{remark}The assumption $p\in\big(1+\tfrac4n,1+\tfrac4{n-2}\big)$ is needed for the scattering result; the lower bound $p>1+4/n$ can be
improved when one only considers the global well-posedness result;
see Remark \ref{re:local} below. Since
we mainly focus the scattering theory, we only consider the case that $p$ belongs to $\big(1+\tfrac4n,1+\tfrac4{n-2}\big)$ .
\end{remark}

\begin{remark} If $a\geq \frac{4}{(p+1)^2}-\lambda_n$ for $n\geq4$ and $a\geq0$ for $n=3$ , the theorem gives
 scattering result for NLS \eqref{Eq3} with all $p\in\big(1+\tfrac4n,1+\tfrac4{n-2}\big)$.
This result is new and allows some negative inverse-square potential when $n\geq4$. Indeed, the restriction on $a$ is from
$a\geq\max\{\frac{4}{(p+1)^2}-\lambda_n,\frac14-\la_n\}$  where the latter is needed in the establishment of interaction Morawetz estimate.
\end{remark}

\begin{remark} The implicit requirement that $a>-\lambda_n$ not only serves for the positivity of the operator $P_a$ but also needs to bound the kinetic energy.
\end{remark}

If the solution $u$ of \eqref{Eq3}  has sufficient decay at infinity
and smoothness, it conserves  mass
\begin{equation}\label{mass}
M(u)=\int_{\mathbb{R}^n}|u(t,x)|^2dx=M(u_0)
\end{equation}
and energy
\begin{equation}\label{energy}
E(u(t))=\tfrac12\int_{\R^n}|\nabla
u(t)|^2dx+\tfrac{a}2\int_{\R^n}\tfrac{|u(t)|^2}{|x|^2}dx+\tfrac1{p+1}\int_{\R^n}|u(t)|^{p+1}dx=E(u_0).
\end{equation}

The paper is organized as follows. In Section $2$,  as a
preliminaries, we give some notations, recall the Strichartz
estimate and prove a generalized Hardy inequality. Section $3$ is
devoted to proving the interaction Morawetz-type estimates for
\eqref{Eq3}.  We show a result about the Sobolev norm equivalence in Section $4$. In
Section 5, we utilize Morawetz-type estimates and the equivalence of
Sobolev norm to prove Theorem \ref{thm}. \vspace{0.2cm}

{\bf Acknowledgments:}\quad  The authors would like to thank Andrew
Hassell and Changxing Miao for their helpful discussions and
encouragement.  They also would like to thank the referee for useful comments. This research was supported by PFMEC(20121101120044), Beijing Natural Science Foundation(1144014),
National Natural Science Foundation of China (11401024) and Discovery Grant DP120102019 from the
Australian Research Council.\vspace{0.2cm}




\section{Preliminaries}

In this section, we first introduce some notation, and then recall the Strichartz estimates and also give two
remarks about the inhomogeneous Strichartz estimate at the endpoint. We conclude this section by showing a generalized
Hardy inequality.

\subsection{Notations}
First, we give some notations which will be used throughout this
paper. To simplify the expression of our inequalities, we introduce
some symbols $\lesssim, \thicksim, \ll$. If $X, Y$ are nonnegative
quantities, we use $X\lesssim Y $ or $X=O(Y)$ to denote the estimate
$X\leq CY$ for some $C$, and $X \thicksim Y$ to denote the estimate
$X\lesssim Y\lesssim X$. We use $X\ll Y$ to mean $X \leq c Y$ for
some small constant $c$. We use $C\gg1$ to denote various large
finite constants, and $0< c \ll 1$ to denote various small
constants. For any $r, 1\leq r \leq \infty$, we denote by $\|\cdot
\|_{r}$ the norm in $L^{r}=L^{r}(\mathbb{R}^n)$ and by $r'$ the
conjugate exponent defined by $\frac{1}{r} + \frac{1}{r'}=1$. We
denote $a_{\pm}$ to be any quantity of the form $a\pm\epsilon$ for
any $\epsilon>0$. We define $\la_n$ by $\la_n=(n-2)^2/4$.

\subsection{Strichartz estimates:}
To state the Strichartz estimate,
we need the following definition

\begin{definition}[Admissible pairs]\label{def1} A pair of exponents $(q,r)$ is called \emph{Schr\"odinger
admissible},  or denote by $(q,r)\in \Lambda_{0}$ if $$2\leq
q,r\leq\infty,~\tfrac{2}{q}=n\big(\tfrac12-\tfrac{1}{r}\big),~\text{and}~(q,r,n)\neq(2,\infty,2).$$

\end{definition}

The Strichartz estimates for the solution of the linear
Schr\"odinger equation have been developed by Burq-Planchon-
Stalker-Tahvildar-Zadeh \cite{BPSS}.

\begin{proposition}[Linear Strichartz estimate \cite{BPSS}]\label{lse} Let $a>-\la_n$ and let
$(q,r)\in\Lambda_0$.
Then there exists a positive constant $C$ depending on $(n,q,r,a)$,
such that
\begin{equation}\label{lstes}
\|e^{itP_a}u_0\|_{L_t^qL_x^r(\R\times\R^n)}\leq C\|u_0\|_{L^2}.
\end{equation}
Furthermore, we have the estimates associated with $P_a$
\begin{equation}\label{lsede}
\big\|P_a^{1/2}\big(e^{itP_a}u_0\big)\big\|_{L_t^qL_x^r(\R\times\R^n)}\leq
C\big\|P_a^{1/2}u_0\big\|_{L^2_x(\R^n)}\simeq \|u_0\|_{\dot H^1}.
\end{equation}

\end{proposition}

By the duality argument and the Christ-Kiselev lemma \cite{CK}, we
obtain the inhomogeneous Strichartz estimates except the endpoint
$(q,r)=(\tilde q,\tilde r)=(2,\tfrac{2n}{n-2})$.
\begin{proposition}[Inhomogeneous Strichartz estimates]\label{prop1}Let
$a>-\la_n$. Suppose $u:I\times\R^n\to\C$ is a solution to
$(i\partial_t+\Delta-\frac{a}{|x|^2})u=f$ with initial data $u_0$.
Then for any $(q,r),~(\tilde q,\tilde r)\in\Lambda_0$ except $(q,r)=(\tilde q,\tilde r)=(2,\tfrac{2n}{n-2})$,
we have
\begin{equation}\label{str1}
\|u\|_{L_t^qL_x^r(I\times\R^n)}\lesssim \|u(t_0)\|_{L^2(\R^n)}+\|f\|_{L_t^{\tilde
q'}L_x^{\tilde r'}(I\times\R^n)},
\end{equation}
and moreover
\begin{equation}\label{instr}
\big\|P_a^{1/2} u\big\|_{L_t^qL_x^r(I\times\R^n)}\lesssim\|u(t_0)\|_{\dot H_x^1(\R^n)}+
    \big\|P_a^{1/2}
    f\big\|_{L_t^{\tilde q'}L_x^{\tilde r'}(I\times\R^n)}.
\end{equation}

\end{proposition}

\begin{remark}\label{rih} Since the dispersive estimate possibly fails when $a<0$ (for wave equation see \cite{PST}), thus one
cannot directly follow Keel-Tao's \cite{KT} argument to obtain the
inhomogeneous Strichartz estimates for the endpoint. However, if
$|a|\leq \epsilon$ where $\epsilon$ is a small enough constant depending on
the Strichartz estimates' constant and the norm
$\||x|^{-2}\|_{L^{n/2,2}}$, then one can prove inhomogeneous
Strichartz estimates at the endpoint. For simplicity, we assume the
initial data $u(t_0)=0$. Indeed, we can write that
\begin{equation}
u(t,x)=\int_0^t e^{i(t-s)P_a} f(s)ds=\int_0^t
e^{i(t-s)\Delta}\left(-a{|x|^{-2}}u+f(s)\right)ds.
\end{equation}

By the endpoint inhomogeneous Strichartz estimates for the classical
Schr\"odinger equation on Lorentz space \cite{KT}, we have
\begin{equation}
\begin{split}
\left\|u\right\|_{L^2_tL^{\frac{2n}{n-2}}}\leq\left\|u\right\|_{L^2_tL^{\frac{2n}{n-2},2}}&\leq
C\left(\epsilon\left\||x|^{-2}u\right\|_{L^2L^{\frac{2n}{n+2},2}}+\|f\|_{L^2L^{\frac{2n}{n+2},2}}\right)\\
&\leq C\left(\epsilon \||x|^{-2}\|_{L^{\frac
n2,2}}\left\|u\right\|_{L^2L^{\frac{2n}{n-2}}}+\|f\|_{L^2L^{\frac{2n}{n+2}}}\right).
\end{split}
\end{equation}
If $\epsilon$ is small such that $C^2\epsilon<1$, then we obtain
the endpoint inhomogeneous Strichartz estimates. We believe that one
can remove the small assumption in the endpoint inhomogeneous Strichartz estimates by following the argument of Hassell-Zhang \cite{HZ} and
considering the spectral measure of Laplacian with inverse square potential on the metric cone.
\end{remark}

\begin{remark} The endpoint inhomogeneous Strichartz estimate
implies the  uniform Sobolev estimate
\begin{equation}
\| (P_a - \alpha)^{-1} \|_{L^r \to L^{r'}} \leq C, \quad r =
\frac{2n}{n+2}, \label{unifSob}\end{equation} where  $C$ is
independent of $\alpha \in \C$. This estimate was proved by Kenig-
Ruiz-Sogge \cite{KRS} for the flat Laplacian without potential, and
by Guillarmou-Hassell\cite{GH} for the Laplacian on nontrapping
asymptotically conic manifolds.  To see this, we choose $w \in
C_c^\infty(\R^n)$ and $\chi(t)$ equal to $1$ on $[-T, T]$ and zero
for $|t| \geq T+1$,
  and let $u(t, x) = \chi(t) e^{i\alpha t} w(x)$.
Then
$$
(i\partial_t + P_a) u = f(t, z), \quad f(t, z) := \chi(t) e^{i\alpha
t} (P_a - \alpha) w(z) + i \chi'(t) e^{i\alpha t} w(z).
$$

Applying the endpoint inhomogeneous Strichartz estimate, we obtain
$$
\| u \|_{L^2_t L^{r'}_z} \leq C \| f \|_{L^2_t L^r_z}.
$$
From the specific form of $u$ and $f$ we have
$$
\| u \|_{L^2_t L^{r'}_z} = \sqrt{2T} \| w \|_{L^{r'}} + O(1), \quad
\|f \|_{L^2_t L^r_z} = \sqrt{2T} \| (P_a - \alpha) w \|_{L^{r}} +
O(1).
$$
Taking the limit $T \to \infty$, we find that
$$
\| w \|_{L^{r'}} \leq C \| (P_a - \alpha) w \|_{L^{r}}.
$$
This implies the uniform Sobolev estimate \eqref{unifSob}. Thus we
have \eqref{unifSob} for $|a|\leq \epsilon$ by previous argument.
\end{remark}

\subsection{The generalized Hardy inequality}
We need the following generalized Hardy inequality:
\begin{lemma}[Hardy inequality] Let $1<p<\infty, 0\leq s<{\frac n p}$, then there exists a constant
$C$ such that for  all $u \in \dot H^s_p(\R^n)$,
\begin{eqnarray}\label{hardy}
\int_{\R^n}\frac{|u(x)|^p}{|x|^{sp}}dx\leq C \|u\|_{\dot H^s_p}^p.
 \end{eqnarray}
\end{lemma}

\begin{remark} This is an improved and extension result of Hardy inequality in  Cazenave \cite{Cav} and Zhang \cite{Zhang}.
The proof heavily relies on the boundedness of the Hardy-Littlewood maximal operator on $L^p$ for $p>1$.
\end{remark}

\begin{proof} It is obvious for $s=0$, hence we only consider $0<s<\frac np<n$.
We begin this proof by recalling the definition of Riesz potential
$I_{\alpha} f$ for $0<\alpha<n$
$$I_{\alpha} f(x)=(-\Delta)^{-\frac{\alpha}2}
f(x)=C_{n,\alpha}\int_{\R^n}|x-y|^{-n+{\alpha}}f(y)dy,$$ where
$C_{n,\alpha}$ is a constant depending on $\alpha$ and $n$. The
norm of homogenous Sobolev space is given by
$$\|f\|_{{\dot
H}^s_p}=\|(-\Delta)^{\frac{s}2} f(x)\|_p=\|I_{-s} f\|_p.$$

Let ${I_{-s} u}=f$, then $u=I_s f$.\;  Thus it suffices to show that
\begin{eqnarray}\label{vhardy}
\bigg\|\frac{I_s f}{|x|^s}\bigg\|_p\leq C\|f \|_p.
\end{eqnarray}
To do so, we write
\begin{align*}
A f(x)&:=\frac{I_s f(x)}{|x|^s}=\int_{\R^n} \frac{f(y)}{|x-y|^{n-s}|x|^s}dy \\
&=\int_{|x-y|\leq 100|x|}\frac{f(y)}{|x-y|^{n-s}|x|^s}dy +
\int_{|x-y|\geq100|x|}\frac{f(y)}{|x-y|^{n-s}|x|^s}dy \\
&=:A_1 f(x) + A_2 f(x).
\end{align*}
To prove \eqref{vhardy}, it suffices to show that both $A_1 $ and
$A_2 $ are strong $(p,p)$ type.

 We first consider $A_1 f$. Notice
that $s>0$, we have
\begin{align*}
A_1 f(x) &=\sum_{j\leq
0}\int_{|x-y|\sim{2^j}100|x|}\frac{|f(y)|}{|x-y|^{n-s}|x|^s}dy\\
&\leq \sum_{j\leq
0}\int_{|x-y|\sim{2^j}100|x|}\frac{|f(y)|}{({2^j}100|x|)^{n-s}|x|^s}dy\\
&\leq\sum_{j\leq
0}\frac{1}{(2^{j}|x|)^n}\int_{|x-y|\leq{2^j}100|x|}|f(y)|dy\cdot
2^{js}\\
&\leq C{\sum_{j\leq 0}2^{js}}Mf(x)\leq C' Mf(x),
\end{align*}
where $M$ is the Hardy-Littlewood maximal operator. By the
boundedness of the Hardy-Littlewood maximal operator for $p>1$, we
obtain $\|A_1 f\|_p\leq C\|f\|_p$.

Next we consider  $A_2 f$. We note that
\begin{equation*}
\begin{split}A_2 f(x)&=\int_{|x-y|\geq100|x|}\frac{f(y)}{|x-y|^{n-s}|x|^s}dy\\&
\leq \int_{|y|\geq{99|x|}}\frac{|f(y)|}{|x-y|^{n-s}|x|^s}dy\leq C\int_{|y|\geq{99|x|}}\frac{|f(y)|}{|y|^{n-s}|x|^s}dy =:B_2
f(x).
\end{split}
\end{equation*}
 For any $g\in\ L^{p'}(\R^n)$, we write
\begin{align*}
\langle B_2 f(x),g(x)\rangle
&=\int_{\R^n}{\int_{|y|\geq{99|x|}}\frac{|f(y)|g(x)}{|y|^{n-s}|x|^s}dy
dx}\\
&=\int_{\R^n}{\frac{1}{|y|^{n-s}}\int_{|x|\leq\frac{|y|}{99}}\frac{g(x)}{|x|^{s}}dx\cdot
|f(y)|dy}\\
&=\langle Tg(y),|f(y)|\rangle,
\end{align*}
where
$$Tg(y)={\frac{1}{|y|^{n-s}}\int_{|x|\leq\frac{|y|}{99}}\frac{g(x)}{|x|^{s}}dx.
}$$ To prove the operator $B_2$ is strong $(p, p)$ type, it is
sufficient to show  by duality
\begin{eqnarray}\label{dual}
\|Tg(y)\|_{p'}\leq C\|g\|_{p'}.
\end{eqnarray}
If we prove $B_2$ is strong $(p, p)$ type, so is $A_2$. Hence we need prove
\eqref{dual}. Since $0<s<n$, we can choose $q>1$ such that ${sq'}<n$. We have $q>\frac{n}{n-s}$. Thus we have by H\"{o}lder's inequality
\begin{align*}
{|Tg(y)|} &\leq
{\frac{1}{|y|^{n-s}}\bigg(\int_{|x|\leq\frac{|y|}{99}}|g(x)|^qdx\bigg)^{\frac{1}{q}}
\bigg(\int_{|x|\leq\frac{|y|}{99}}\frac{1}{|x|^{sq'}}dx\bigg)^{\frac{1}{q'}}} \\
&\leq C\frac{1}{|y|^{n-s}}\cdot\bigg(\int_{|x|\leq\frac{|y|}{99}}|g(x)|^qdx\bigg)^{\frac{1}{q}}\cdot{|y|^{{(n-{sq'})}\frac{1}{q'}}}\\
&\leq C\frac{1}{|y|^{\frac{n}{q}}}\cdot\bigg(\int_{|x-y|\leq2|y|}|g(x)|^qdx\bigg)^{\frac{1}{q}}\\
&\leq C(M(|g|^q))^{\frac{1}{q}}(y).
\end{align*}
For all $p'>q>{\frac{n}{n-s}}$,  one has $1<p<n/s$. Since $p'>q$,
the boundedness of Hardy-Littlewood maximal operator gives
$$\left\|(M(|g|^q))^\frac 1 q\right\|_{L^{p'}}=\left\|M(|g|^q)\right\|_{L^{\frac {p'} q
}}^\frac{1}{q}\leq C \left\||g|^q\right\|_{L^{\frac {p'}
q}}^{\frac{1}{q}} = C \|g\|_{L^{p'}}. $$ Therefore we obtain $T$ is
strong $(p',p')$ type, hence we proves \eqref{dual}. Thus we conclude the proof of this lemma.

\end{proof}




\section{Morawetz-type estimates}
In this section, we derive the quadratic Morawetz identity for
\eqref{Eq3} and then establish the interaction Morawetz estimate.
The Morawetz estimate provides us a decay of the solution to the NLS
with an inverse square potential, which will help us study the asymptotic
behavior of the solutions in the energy space in next section. More precisely, we have

\begin{proposition}[Morawetz-type estimates]\label{Morawetz}
 Let $u$ be an $H^\frac12$-solution to
\eqref{Eq3} on the spacetime slab $I\times\R^n$, the dimension
$n\geq3$ and $a>\tfrac14-\la_n$, then we have
\begin{equation}\label{1.2}
\big\||\nabla|^{\frac{3-n}2}(|u|^2)\big\|_{L^2(I;L^2(\R^n))}\leq
C\|u(t_0)\|_{L^2}\sup_{t\in I}\|u(t)\|_{\dot H^{\frac12}},~t_0\in I,
\end{equation}
and hence
\begin{equation}\label{interac2}
\big\||\nabla|^{\frac{3-n}4}u\big\|_{L_t^4(I;L_x^4(\R^n))}^2\leq
C\|u(t_0)\|_{L^2}\sup_{t\in I}\|u(t)\|_{\dot H^{\frac12}}.
\end{equation}
\end{proposition}

\begin{remark}When $n=3$, it is obvious to see that the result also holds for
$a=0$; see \cite{CKSTT}.
\end{remark}

\begin{proof}
We consider the NLS equation in the form of
\begin{equation}\label{NLS}
i\partial_tu+\Delta u=gu
\end{equation}
where $g=g(\rho,|x|)$ is a real function of $\rho=|u|^2=2T_{00}$ and
$|x|$.  We first recall the conservation laws for free Schr\"odinger
in Tao \cite{Tao}
\begin{equation*}
\begin{split}
\partial_t T_{00}+\partial_j T_{0j}=0,\\
\partial_t T_{0j}+\partial_k T_{jk}=0,
\end{split}
\end{equation*}
where the mass density quantity $T_{00}$ is defined by
$T_{00}=\tfrac12|u|^2,$ the mass current and the momentum density
quantity $T_{0j}=T_{j0}$ is given by $T_{0j}=T_{j0}=\mathrm{Im}(\bar
u\partial_j u)$, and the quantity $T_{jk}$ is
\begin{equation}\label{stress}
T_{jk}=2\mathrm{Re}(\partial_j u
\partial_k\bar u)-\tfrac12\delta_{jk}\Delta(|u|^2),
\end{equation}
for all $j,k=1,...n,$ and $\delta_{jk}$ is the Kroncker delta. Note
that the kinetic terms are unchanged, we see that for \eqref{NLS}
\begin{equation}\label{Local Conservation}
\begin{split}
\partial_t T_{00}+\partial_j T_{0j}&=0,\\
\partial_t T_{0j}+\partial_k T_{jk}&=-\rho\partial_j g.
\end{split}
\end{equation}

By the density argument, we may assume sufficient smoothness and
decay at infinity of the solutions to the calculation and in
particular to the integrations by parts. Let $h$ be a sufficiently
regular real even function defined in $\R^n$, e.g. $h=|x|$. The
starting point is the auxiliary quantity
\begin{equation*}
J=\tfrac12\langle|u|^2, h\ast |u|^2\rangle=2\langle T_{00}, h\ast
T_{00}\rangle.
\end{equation*}
Define the quadratic Morawetz quantity $M=\tfrac14\partial_t J$.
Hence we can precisely rewrite
\begin{equation}\label{3.1}
M=-\tfrac12\langle\partial_jT_{0j}, h\ast
T_{00}\rangle-\tfrac12\langle T_{00}, h\ast
\partial_jT_{0j} \rangle=-\langle T_{00}, \partial_j h\ast
T_{0j} \rangle.
\end{equation}
By \eqref{Local Conservation} and integration by parts, we have
\begin{equation*}
\begin{split}
\partial_tM&=\langle\partial_kT_{0k}, \partial_j h\ast T_{0j} \rangle-\langle T_{00},
\partial_j h\ast\partial_t T_{0j} \rangle\\&=-\sum_{j,k=1}^n\langle T_{0j}, \partial_{jk} h\ast T_{0j} \rangle+\langle T_{00},
\partial_{jk} h\ast T_{jk} \rangle+\langle \rho,
\partial_j h\ast(\rho\partial_j g) \rangle.
\end{split}
\end{equation*}
For our purpose, we note that
\begin{equation}
\begin{split}
\sum_{j,k=1}^n\langle T_{0k},  \partial_{jk} h\ast T_{0j}
\rangle&=\big\langle \mathrm{Im}(\bar u\nabla u), \nabla^2 h\ast
\mathrm{Im}(\bar u\nabla u) \big\rangle\\&=\big\langle \bar u\nabla
u, \nabla^2 h\ast \bar u\nabla u \rangle-\langle \mathrm{Re}(\bar
u\nabla u), \nabla^2 h\ast \mathrm{Re}(\bar u\nabla u) \big\rangle.
\end{split}
\end{equation}
Therefore it yields that
\begin{equation*}
\begin{split}
\partial_tM=&\big\langle \mathrm{Re}(\bar u\nabla
u), \nabla^2 h\ast \mathrm{Re}(\bar u\nabla u)
\big\rangle-\big\langle \bar u\nabla u, \nabla^2 h\ast \bar u\nabla
u \big\rangle\\&+\Big\langle \bar uu,
\partial_{jk} h\ast \big(\mathrm{Re}(\partial_j u \partial_k\bar
u)-\tfrac14\delta_{jk}\Delta(|u|^2)\big) \Big\rangle+\big\langle
\rho,
\partial_j h\ast(\rho\partial_j g) \big\rangle.
\end{split}
\end{equation*}
From the observation
\begin{equation*}
\begin{split}
-\big\langle \bar uu,
\partial_{jk} h\ast\delta_{jk}\Delta(|u|^2) \big\rangle=\big\langle \nabla (|u|^2), \Delta h\ast
\nabla(|u|^2) \big\rangle,
\end{split}
\end{equation*}
we write
\begin{equation}\label{Morawetz equality}
\begin{split}
\partial_tM=\tfrac12\langle \nabla \rho, \Delta h\ast\nabla\rho \rangle+R+\big\langle \rho,
\partial_j h\ast(\rho\partial_j g) \big\rangle,
\end{split}
\end{equation}
where $R$ is given by
\begin{equation*}\label{3.4}
\begin{split}
R&=\big\langle \bar uu, \nabla^2 h\ast (\nabla\bar u \nabla u)
\big\rangle-\big\langle \bar u\nabla u, \nabla^2 h\ast \bar u\nabla
u \big\rangle\\&=\tfrac12\int \Big(\bar u(x)\nabla \bar u(y)-\bar
u(y)\nabla\bar u(x)\Big)\nabla^2h(x-y)\Big(u(x)\nabla
u(y)-u(y)\nabla u(x)\Big)\mathrm{d}x\mathrm{d}y.
\end{split}
\end{equation*}

Since the Hessian of $h$ is positive definite, we have $R\geq0$. Integrating
over time in an interval $[t_1, t_2]\subset I$ yields
\begin{equation*}
\begin{split}
\int_{t_1}^{t_2}\Big\{\frac12\langle \nabla \rho, \Delta
h\ast\nabla\rho \rangle+\langle \rho,
\partial_j h\ast(\rho\partial_j g) \rangle+R\Big\}\mathrm{d}t=-\langle T_{00}, \partial_j h\ast
T_{0j} \rangle\big|_{t=t_1}^{t=t_2}.
\end{split}
\end{equation*}

From now on, we choose $h(x)=|x|$. One can follow the arguments in
\cite{CKSTT} to bound the right hand by the quantity
\begin{equation*}
\Big|\mathrm{Im}\int_{\R^{2n}}|u(x)|^2\frac{x-y}{|x-y|}\bar
u(y)\nabla u(y)dxdy\Big|\leq C\sup_{t\in
I}\|u(t)\|^2_{L^2}\|u(t)\|^2_{\dot H^{\frac12}}.
\end{equation*}
Therefore we conclude
\begin{equation}\label{Morawetz inequality}
\int_{t_1}^{t_2}\big\langle \rho,
\partial_j h\ast(\rho\partial_j g) \big\rangle dt+\big\||\nabla|^{\frac{3-n}2}(|u|^2)\big\|_{L^2(I;L^2(\R^n))}\leq C\sup_{t\in
I}\|u(t)\|_{L^2}\|u(t)\|_{\dot H^{\frac12}}.
\end{equation}

Now we consider the term
\begin{equation*}
\begin{split}
P&:=\big\langle \rho, \nabla h\ast (\rho\nabla g) \big\rangle.
\end{split}
\end{equation*}
Consider $g(\rho,|x|)=\rho^{(p-1)/2}+V(x)$, then we can write
$P=P_1+P_2$ where
\begin{equation}\label{3.5}
\begin{split}
P_1= \big\langle\rho, \nabla h\ast \big(\rho\nabla
(\rho^{(p-1)/2})\big)\big\rangle=\frac{p-1}{p+1}\big\langle\rho,
\Delta h\ast \rho^{(p+1)/2}\big\rangle\geq 0
\end{split}
\end{equation}
and
\begin{equation}\label{P2}
\begin{split}
P_2= \iint\rho(x)\nabla h(x-y)\rho(y)\nabla
\big(V(y)\big)\mathrm{d}x\mathrm{d}y.
\end{split}
\end{equation}

Comparing \eqref{Eq3} and \eqref{NLS}, we see $V(x)=a|x|^{-2}$. We claim that
\begin{equation*}
\begin{split}
\Big|\int_{t_1}^{t_2}P_2dt\Big| =&
\Big|2a\int_{t_1}^{t_2}\iint|u(x)|^2\frac{(x-y)}{|x-y|}\cdot
y|y|^{-4}|u(y)|^2\mathrm{d}x\mathrm{d}ydt\Big|\\
\lesssim&
\sup_{t\in I}\|u_0\|_{L^2}^2\|u(t)\|_{\dot H^{\frac12}}^2.
\end{split}
\end{equation*}
To show this, it suffices to show
\begin{equation}
\begin{split}
\int_{t_1}^{t_2}\int|x|^{-3}|u(t,x)|^2\mathrm{d}xdt\lesssim
\sup_{t\in I}\|u(t)\|_{\dot H^{\frac12}}^2.
\end{split}
\end{equation}

Let $A=\partial_r+\frac{n-1}{2r}$ and $H=\Delta-a|x|^{-2}$, where
$r=|x|$. Now we consider the quantity $\langle Au,u\rangle$. Since
$u_t=i(Hu-f(u))$ with $f(u)=|u|^{p-1}u$, we have
\begin{equation*}
\begin{split}
\partial_t\langle Au,u\rangle= i\big\langle [A,H]u,u\big\rangle+i\big(\langle
Au,f(u)\rangle-\langle Af(u),u\rangle\big).
\end{split}
\end{equation*}

We first consider the term from the nonlinear part
\begin{equation*}
\begin{split}
\Big(\langle Au,f(u)\rangle-\langle Af(u),u\rangle\Big)=\Big(\langle
\partial_ru,f(u)\rangle+\langle f(u),\partial_ru\rangle+\langle f(u),\tfrac{n-1}{|x|}u\rangle\Big).
\end{split}
\end{equation*}
By assuming $u$ rapidly tends to zero as $r\rightarrow\infty$, we
obtain
\begin{align}\nonumber
\Big(\langle Au,f(u)\rangle-\langle
Af(u),u\rangle\Big)&=\frac{2}{p+1}\int_{\R^n}\partial_r\big(|u|^{p+1}(x)\big)dx+\int_{\R^n}\frac{(n-1)|u|^{p+1}}{|x|}dx
\\&=(n-1)\frac{p-1}{p+1}\int_{\R^n}\frac{|u|^{p+1}}{|x|}dx.
\end{align}

Now we consider the term from linear part. Note that the
commutator
\begin{equation}[A,H]=
\begin{cases}
-2\Delta_{\mathbb{S}^{n-1}}r^{-3}+c\delta+2ar^{-3}\qquad n=3;\\
-2\Delta_{\mathbb{S}^{n-1}}r^{-3}+\frac12(n-1)(n-3)r^{-3}+2ar^{-3}\qquad\qquad n\geq4.\\
\end{cases}
\end{equation}
where the constant $c>0$ and $\delta$ is the delta function.
Integrating on a finite time interval $[t_1,t_2]$, we have for
$n=3$
\begin{equation}\label{n3}
\begin{split}
i^{-1}\langle
Au,u\rangle\big|_{t_1}^{t_2}=&2\int_{t_1}^{t_2}\int_{\R^n}\frac{|\nabla_\theta
u|^2}{r^3}dxdt+2\int_{t_1}^{t_2}|u(t,0)|^2dt\\&+2a\int_{t_1}^{t_2}\int_{\R^n}\frac{|u|^2}{r^3}dxdt
\\&+\frac{2(p-1)}{p+1}\int_{t_1}^{t_2}\int_{\R^n}\frac{|u|^{p+1}}{r}dx dt
\end{split}
\end{equation}
and for $n\geq4$
\begin{equation}\label{n4}
\begin{split}
i^{-1}\langle
Au,u\rangle\big|_{t_1}^{t_2}=&2\int_{t_1}^{t_2}\int_{\R^n}\frac{|\nabla_\theta
u|^2}{r^3}dxdt\\&+\big[\frac12(n-1)(n-3)+2a\big]\int_{t_1}^{t_2}\int_{\R^n}\frac{|u|^2}{r^3}dxdt
\\&+(n-1)\frac{p-1}{p+1}\int_{t_1}^{t_2}\int_{\R^n}\frac{|u|^{p+1}}{r}dxdt.
\end{split}
\end{equation}
Note $a>\tfrac14-\la_n$, the constant before the quantity
$\int_{t_1}^{t_2}\int_{\R^n}\frac{|u|^2}{r^3}dxdt$ is strictly
positive. From \eqref{n3} and \eqref{n4}, we have by
interpolation and Hardy's inequality
\begin{equation*}
\begin{split}
\int_{t_1}^{t_2}\int\frac{|u(t,x)|^2}{|x|^3}\mathrm{d}xdt\lesssim_a
\sup_{t\in[t_1,t_2]}\bigg(\Big|\int_{\R^n} \partial_ru \bar{u}
dx\Big|+\int_{\R^n} \frac{|u|^2}{|x|} dx\bigg)\lesssim_a \sup_{t\in
[t_1,t_2]}\|u(t)\|_{\dot H^{\frac12}}^2.
\end{split}
\end{equation*}
Therefore, we conclude the proof of Proposition \ref{Morawetz}.

\end{proof}

\begin{remark}\label{rem1.6} Our strategy can be used to prove the same estimates
for the Schr\"odinger equation with potential $a|x|^{\sigma}$ where
$a\leq0$ and $\sigma>1$. In particular $\sigma=2$, we prove the interaction Morawetz estimates for defocusing NLS with repulsive harmonic potential.
This possibly is used to remove the radial assumption in the repulsive case of \cite{KVZ}, which studied the scattering theory of the energy-critical defocusing NLS with harmonic potential.
In contrast to the inverse square potential, the error term brought by the
potential is positive. Indeed, we consider $V(x)=a|x|^{\sigma}$ in \eqref{P2}. From the observation
$\nabla h$ is odd, it follows that
\begin{equation*}
\begin{split}
P_2=\int\rho(x)\rho(y)\nabla h(y-x)\nabla
\big(V(x)\big)\mathrm{d}x\mathrm{d}y=-\int\rho(x)\rho(y)\nabla
h(x-y)\nabla\big(V(x)\big)\mathrm{d}x\mathrm{d}y.
\end{split}
\end{equation*}
Thus we can write
\begin{equation}
\begin{split}
P_2=\frac12\int\rho(x)\rho(y)\nabla h(x-y)\cdot\big[\nabla
\big(V(y)\big)-\nabla \big(V(x)\big)\big]\mathrm{d}x\mathrm{d}y.
\end{split}
\end{equation}
By the mean value theorem, we easily see that
\begin{align*}
\nabla h(x-y)\cdot\big(\nabla V(y)-\nabla V(x)\big)=&\nabla
h(x-y)\cdot\Big(\nabla V(y)-\nabla V\big((x-y)+y\big)\Big)\\
=&-\nabla
h(x-y)\cdot\int_0^1(x-y)\cdot\nabla^2V\big(y+\theta(x-y)\big)d\theta\\
=&-|x-y|^{-1}\int_0^1(x-y)\otimes(x-y)\nabla^2V\big(y+\theta(x-y)\big)d\theta.
\end{align*}
We see that $\nabla^2V$ is negative when $V(x)=a|x|^\sigma$ with $a<
0$ and $\sigma>1$ by using
$$\nabla^2(|x|^{\sigma})=\sigma|x|^{\sigma-2}\bigg(I_{n\times
n}+(\sigma-2)\frac{(x_1,\cdots,x_n)^T}{|x|}\cdot\frac{(x_1,\cdots,x_n)}{|x|}\bigg),$$
whose the $k$-order principal minor determinant is
$$\bigg|I_{k\times k}+(\sigma-2)\Big(\frac{x_1}{|x|},\cdots,
\frac{x_k}{|x|}\Big)^T\cdot \Big(\frac{x_1}{|x|},\cdots,
\frac{x_k}{|x|}\Big)\bigg|=1+(\sigma-2)\frac{x_1^2+\cdots+x_k^2}{|x|^2}>0.$$
Hence we can check that $P_2$ is nonnegative when
$V(x)=a|x|^{\sigma}$ with $a\leq 0$ and $\sigma>1$. This together
with \eqref{Morawetz inequality} and \eqref{3.5} concludes the proof of
the case $\sigma>1$. \vspace{0.2cm}

\end{remark}




\section{Sobolev norm equivalence}

In this section, we  study the equivalence of the Sobolev norms based on the
operator $P_a$ and the standard Sobolev norms based on the
Laplacian. For simplicity, we define
\begin{equation}\label{r0r1}
r_0=\tfrac{2n}{\min\{n+2+\sqrt{(n-2)^2+4a},2n\}},\quad
r_1=\tfrac{2n}{\max\{n-\sqrt{(n-2)^2+4a},0\}}.
\end{equation}
The purpose of this section is to prove the following
\begin{proposition}[Sobolev norm equivalence]\label{equiv}
Let $n\geq3,~a>-\la_n$. Then, there exist constants $C_1,~C_2>0$
depending on $(n,a,s)$ such that

$\bullet$ when $s=1$ and $r\in(r_0,r_0')\cap (r'_1,r_1)\cap (\frac
n{n-1},n)$
\begin{equation}\label{equinorm1}
C_1\|f\|_{\dot H_r^1(\R^n)}\leq\|P_a^{1/2}f\|_{L^r(\R^n)}\leq
C_2\|f\|_{\dot H_r^1(\R^n)},~~\forall~f\in\dot H_r^1(\R^n);
\end{equation}

$\bullet$ if $a\geq0$, for $0\leq s\leq 1$ and $r\in(r_0'/(r_0'-s),
r_1/s)\cap (1,n/s)$
\begin{equation}\label{equinorm2}
C_1\|f\|_{\dot H_r^s(\R^n)}\leq\|P_a^{s/2}f\|_{L^r(\R^n)}\leq
C_2\|f\|_{\dot H_r^s(\R^n)},~~\forall~f\in\dot H_r^s(\R^n).
\end{equation}
\end{proposition}

\begin{remark} This result generalizes the equivalent result of the Sobolev norms in \cite{BPSS}, which proved $\|f\|_{\dot H^s(\R^n)}\sim\|P_a^{s/2}f\|_{L^2(\R^n)}$
for $-1\leq s\leq 1$ and $a>-\la_n$.
\end{remark}
\begin{remark} From the classical Sobolev result, we know that \eqref{equinorm1} with $a=0$ holds for $1<r<\infty$.
However, it is easy to check that the interval $(r_0,r_1)\rightarrow (1,n)$ as $a\rightarrow0$. It was pointed out to the authors by Andrew Hassell
that the dependence on $a$ in the boundedness of Riesz transform is not continuous, which means that the equivalent norm result is strongly influenced by the inverse square potential even though for sufficiently small $|a|$.
\end{remark}
This proposition follows from

\begin{proposition}\label{pequiv} There exist constants $c_r$ and $C_r$ satisfying the following estimates:

$\bullet$ If $a>-\la_n$, we have for $r\in (r_0,r_1)$
\begin{equation}\label{Riesz}
\|\nabla f\|_{L^r}\leq C_r\| P_a^{\frac 12}f\|_{L^r};
\end{equation}
In addition, the reverse estimate holds for $r\in (r_0,r_1)$ and
$1<r,r'<n$
\begin{equation}\label{RRiesz}
\|P_a^{\frac 12}f\|_{L^{r'}}\leq c_r\|\nabla f\|_{L^{r'}}.
\end{equation}

$\bullet$ If $a\geq0$, we have for $0\leq s\leq 2$ and $r\in
(1,n/s)$
\begin{equation}\label{norm}
\|P_a^{\frac s2}f\|_{L^r}\leq C_r\||\nabla|^{s}f\|_{L^r}
\end{equation}
In addition, the reverse estimate holds for $0\leq s\leq 1$ and
$r\in(r_0'/(r_0'-s), r_1/s)$
\begin{equation}\label{rnorm}
\|P_a^{\frac s2}f\|_{L^r}\geq c_r\||\nabla|^{s}f\|_{L^r}
\end{equation}

\end{proposition}
For $a\geq0$, since the heat kernel $e^{-tP_a}$ satisfies the
Gaussian upper bounds, we can follow
D'Ancona-Fanelli-Vega-Visciglia's \cite{DFVV} argument, which are in
spirt of Sikora-Wright \cite{SW} proving the boundedness of
imaginary power operators and Stein-Weiss complex interpolation
theorem. When $a<0$, the Gaussian upper bound of the heat kernel
fails, we have to resort to the boundedness of the Riesz transform,
which was proved in Hassell-Lin \cite{HL}. That is why we have to
restrict $s=1$. However, this is enough for considering the
wellposedness and scattering theory in energy space $H^1$. We
believe that one could establish the equivalence of the Sobolev norm
on metric cone and perturbated by the inverse square potential,
which is a topic we plan to address in future articles.

\begin{proof}

We first consider \eqref{Riesz}, which is a consequence of the
boundedness of Riesz transform $\nabla P_a^{-1/2}$. The theorem of
Hassell-Lin \cite{HL} established $L^r$-boundness of Riesz transform
of Schr\"odinger operator with inverse square potential on a metric
cone. The result implies that if $r\in (r_0,r_1)$ where $r_0,r_1$
are defined in \eqref{r0r1}, then Riesz transform $\nabla
P_a^{-\frac12}$ is bounded on $L^r$.

 Next we use a duality argument
to show \eqref{RRiesz}. Since $P_a^{\frac12}C_0^\infty$ is dense in
$L^r$ (see \cite[Appendix]{Russ}), then
$$\|P_a^{\frac12} f\|_{L^{r'}}=\sup\Big\{\langle P_a^{\frac12}f,P_a^{\frac12}g\rangle: g\in C_0^\infty(\R^n),\|P_a^{\frac12}g\|_{L^r}\leq 1\Big\}.$$
Therefore by the definition of the square root of $P_a$, we see
\begin{equation*}
\begin{split}
\|P_a^{\frac12} f\|_{L^{r'}}&\leq |\langle
P_a^{\frac12}f,P_a^{\frac12}g\rangle|=|\langle P_af,g\rangle|\\&\leq
\|\nabla f\|_{L^{r'}}\|\nabla g\|_{L^{r}}+\left\|
f/{|x|}\right\|_{L^{r'}}\left\|g/{|x|}\right\|_{L^{r}}.
\end{split}
\end{equation*}
If $1<r,r'<n$, the Hardy inequality \eqref{hardy} (see below)
implies
\begin{equation}
\begin{split}
\|P_a^{\frac12} f\|_{L^{r'}}\leq C\|\nabla f\|_{L^{r'}}\|\nabla
g\|_{L^{r}}.
\end{split}
\end{equation}
By \eqref{Riesz}, hence we prove \eqref{RRiesz}.\vspace{0.2cm}

Next we prove \eqref{norm} and \eqref{rnorm}. We first recall
two-side estimates of the heat kernel associated to the operator
$P_a$, which were found independently by Liskevich-Sobol
\cite[Remarks at the end of Sec. 1]{LS} and Milman-Semenov
\cite[Theorem 1]{MS}.

\begin{lemma}[Heat kernel boundedness]\label{lem:heat} Assume $a>-\la_n$ and let
$H(t,x,y)$ be the kernel of the operator $e^{-tP_a}$. Then there
exist positive constants $C_1,C_2$ and $c_1,c_2$ such that for all
$t>0$ and all $x,y\in \R^n\setminus\{0\}$
\begin{equation}\label{heat}
\begin{split}
C_1&\varphi_\sigma(x,t)\varphi_\sigma(y,t)t^{-\frac
n2}\exp\left(-|x-y|^2/c_1t\right)\leq H(t,x,y)\\&\leq
C_2\varphi_\sigma(x,t)\varphi_\sigma(y,t)t^{-\frac
n2}\exp\left(-|x-y|^2/c_2t\right),
\end{split}
\end{equation}
where the weight function
\begin{equation}\label{weight}
\varphi_{\sigma}(x,t)=\begin{cases}\left(\frac{\sqrt
t}{|x|}\right)^{\sigma}\quad &\mathrm{if}~|x|\leq \sqrt{t},\\\quad
1\quad &\mathrm{if}~|x|\geq \sqrt{t}\\
\end{cases}
\end{equation}
and $\sigma=\sigma(a)=\frac12(n-2)-\frac12\sqrt{(n-2)^2+4a}$.
\end{lemma}
\begin{remark} We notice that $\sigma(a)\leq 0$ for $a\geq0$ and
$0<\sigma(a)<(n-2)/2$ for $a\in (-\la_n,0)$.
\end{remark}
We need a theorem of Sikora-Wright \cite{SW} that established a weak
type estimate for imaginary powers of self-adjoint operator, defined
by spectral theory. The result implies that if the heat kernel
$H(t,x,y)$ associated with the operator $\mathrm{H}$ satisfies that
$$H(t,x,y)\lesssim t^{-\frac
n2}\exp\left(-|x-y|^2/c_2t\right),
$$ then for all $y\in \R$ the imaginary powers $\mathrm{H}^{iy}$ is
weak-$(1,1)$. By Lemma \ref{lem:heat} and the well known Gaussian
upper for heat kernel $e^{-t\Delta}$, the operators $P_a^{iy}$
$(a\geq0)$ and $(-\Delta)^{-iy}$ satisfy weak-$(1,1)$ type estimate
of $O(1+|y|)^{n/2}$. On the other hand, the operators $P_a^{iy}$ and
$(-\Delta)^{-iy}$ are obviously bounded on $L^2$ by the spectral
theory on Hilbert space. Hence the operators $P_a^{iy}$ and
$(-\Delta)^{-iy}$ are bounded on $L^r$ for all $1<r<\infty$.

Now we define the analytic family of operators for $z\in\C$
\begin{equation}
T_z=P_a^z(-\Delta)^{-z}
\end{equation}
where $P_a^z=\int_0^\infty\lambda^z dE_{\sqrt{P_a}}(\lambda)$ and
$(-\Delta)^z$ is defined by the Fourier transform. Writing $z=x+iy$
for $x\in [0,1]$, we have
\begin{equation*} T_z=P_a^{iy}P_a^{x}(-\Delta)^{-x}(-\Delta)^{-iy},
\quad y\in\R,~ x\in [0,1].
\end{equation*}

When $\Re z=0$, then we have for all $1<p_0<\infty$
\begin{equation}\label{x0}
\|T_z\|_{L^{p_0}\rightarrow L^{p_0}}=\|T_{iy}\|_{L^{p_0}\rightarrow
L^{p_0}}\leq C(1+|y|)^{n(1-\frac2{p_0})}.
\end{equation}
Notice $P_af=-\Delta
f+a|x|^{-2}f$, it follows from the Hardy inequality \eqref {hardy}
that
\begin{equation*}
\left(\int_{\R^n}|P_af|^{p}dx\right)^{\frac1{p}}\leq
C\left(\int_{\R^n}|\Delta f|^{p}dx\right)^{\frac1{p}},\quad
1<{p}<\frac n2.
\end{equation*}
Therefore for $\Re z=1$, we have by the $L^p$-boundedness of
$P_a^{iy}$ and $(-\Delta)^{-iy}$
\begin{equation}\label{x1}
\|T_{1+iy}\|_{L^{p_1}\rightarrow
L^{p_1}}\leq\|P_a(-\Delta)^{-1}\|_{L^{p_1}\rightarrow L^{p_1}}\leq
C(1+|y|)^{n(1-\frac2{p_1})},  1<{p_1}<\frac n2.
\end{equation}
Applying complex interpolation to \eqref{x0} and \eqref{x1}, we
obtain for real number $\sigma\in[0,1]$
\begin{equation}
\|T_{\sigma}\|_{L^r\rightarrow L^r}\leq C,
\end{equation}
where $1/r=(1-\sigma)/p_0+\sigma/p_1$ with $1<p_0<\infty$ and
$1<p_1<n/2$. This gives $r\in(1,n/(2\sigma))$, hence $1<r<n/s$ which
proves \eqref{norm}.\vspace{0.1cm}

Finally we prove \eqref{rnorm}. We similarly define the analytic
family of operators for $z\in\C$
\begin{equation}
\widetilde{T}_z=(-\Delta)^{z}P_a^{-z}.
\end{equation}
Writing $z=x+iy$ for $x\in [0,1/2]$, we have
\begin{equation} \widetilde{T}_z=(-\Delta)^{iy}(-\Delta)^{x}P_a^{-x}P_a^{-iy},
\quad y\in\R,~ x\in [0,1/2].
\end{equation}

As before the operators $P_a^{-iy}$ and $(-\Delta)^{iy}$ are bounded
on $L^r$ for all $1<r<\infty$. On the other hand, we can use the
dual argument  as above and boundedness of classical Riesz transform
$\nabla(-\Delta)^{-\frac12}$ to show
$$\|(-\Delta)^{\frac12}f\|_{L^r}\leq C\|\nabla f\|_{L^r},\quad 1<r<\infty.$$ By \eqref{Riesz}, we hence have
\begin{equation}\label{hx}
\|\widetilde{T}_z\|_{L^{p_1}\rightarrow L^{p_1}}\leq C,\quad
\text{for}~p_1\in(r_0,r_1), \quad\Re z=1/2.
\end{equation}

When $\Re z=0$, then we have for all $1<p_0<\infty$
\begin{equation}\label{hx0}
\|\widetilde{T}_z\|_{L^{p_0}\rightarrow
L^{p_0}}=\|T_{iy}\|_{L^{p_0}\rightarrow L^{p_0}}\leq
C(1+|y|)^{n(1-\frac2{p_0})}.
\end{equation} By the Stein-Weiss interpolation, we obtain
for real number $\sigma\in[0,1/2]$
\begin{equation}
\|\widetilde{T}_{\sigma}\|_{L^r\rightarrow L^r}\leq C,
\end{equation}
where $1/r=(1-2\sigma)/p_0+2\sigma/{p_1}$ with $1<p_0<\infty$, and
$p_1\in(r_0,r_1)$. This implies $r\in(r_0'/(r_0'-2\sigma),
r_1/(2\sigma))$. Note $\sigma=s/2$, we prove \eqref{rnorm}.

\end{proof}




\section{Proof of Theorem \ref{thm}}
In this section, we prove Theorem \ref{thm}. The key points are the Strichartz estimate, Leibniz rule obtained by the equivalence Sobolev norm
and the Morawetz type estimate.

\subsection{Global well-posedness theory}
By the mass and energy
conservation, the global well-posedness follows from

\begin{proposition}[Local well-posedness theory]\label{localwell}
Let $n\geq 3$ and $1+\frac2{n-2}\leq
p<1+\frac4{n-2}$. Assume that  $a>-\frac4{(p+1)^2}\la_n$ and $u_0\in H^1(\R^n)$. Then
there exists $T=T(\|u_0\|_{H^1})>0$ such that the equation
\eqref{Eq3} with initial data $u_0$ has a unique solution $u$ with
\begin{equation}\label{small}
u\in C(I; H^1(\R^n))\cap L_t^{q}(I; H^1_r(\R^n)),\quad I=[0,T),
\end{equation}
where the pair $(q,r)\in\Lambda_0$ satisfies
\begin{equation}
(q,r)=\begin{cases}
\left(\tfrac{4(p+1)}{(n-2)(p-1)},\tfrac{n(p+1)}{n+p-1}\right),\quad
\text{if}\quad
a\geq0;\\
\left(q,p\left(\frac1{(2n/(n+2))_+}+\frac{p-1}n\right)^{-1}\right),\quad \text{if}\quad\min\{1-\lambda_n,0\}\leq a<0;\\
\left(q,p\left(\frac1{(r_1')_+}+\frac{p-1}n\right)^{-1}\right),\quad\text{if}\quad
-\frac4{(p+1)^2}\lambda_n<a<\min\{1-\lambda_n,0\}.
\end{cases}
\end{equation}
Here $r_1$ is defined in \eqref{r0r1}.
\end{proposition}
\begin{remark}\label{re:local}One can release the restriction $p\geq1+\frac2{n-2}$ to $p>1$ if $a\geq0$. If $a<0$, one can also
improve the range of $p$ which depends on $a$. We do not give the detail since this result is enough for showing the scattering.
\end{remark}
\begin{proof}
We follow the standard Banach fixed point argument to prove this
result. To this end, we consider the map
\begin{equation}\label{inte3}
\Phi(u(t))=e^{itP_a}u_0-i\int_{0}^te^{i(t-s)P_a}(|u|^{p-1}u(s))ds
\end{equation}
on the complete metric space $B$
\begin{align*}
B:=\big\{&u\in Y(I)\triangleq C_t(I; H^{1})\cap L_t^{q}(I;
H^1_r):\ \|u\|_{Y(I)}\leq2CC_1\|u_0\|_{H^1}\big\}
\end{align*}
with the metric
$d(u,v)=\big\|u-v\big\|_{L_t^{q}L_x^{r}(I\times\R^n)}$. We need to
prove that the operator $\Phi$ defined by $(\ref{inte3})$
is well-defined on $B$ and is a contraction map under the metric $d$
for $I$.\vspace{0.1cm}

To see this, we first consider the case $a\geq0$. Therefore we have by Proposition \ref{equiv} for $n\geq3$
\begin{equation}\label{eqnorm}
\|\nabla f\|_{L^r}\simeq C_r\| P_a^{\frac 12}f\|_{L^r},\quad
\forall~ r\in(1,n).
\end{equation}

Let
$(q,r)=\big(\tfrac{4(p+1)}{(n-2)(p-1)},\tfrac{n(p+1)}{n+p-1}\big)$.
It is easy to verify that $(q,r)\in\Lambda_0$ such that
$r,~{r}'\in(1,n)$. Then we have by Strichartz estimate and
\eqref{eqnorm}
\begin{align*}
\big\|\Phi(u)\big\|_{Y(I)}\leq&C_1\big\|\langle
P_a^{1/2}\rangle\Phi(u)\big\|_{
L_t^q(I;
L_x^r)}\\
\leq& CC_1\|u_0\|_{H^1}+CC_1\big\|\langle P_a^{1/2}\rangle
(|u|^{p-1}u) \big\|_{L_t^{{q}'}L_x^{{r}'}(I\times\R^n)
}\\
\leq&CC_1\|u_0\|_{H^1}+CC_1C_2\big\|\langle\nabla\rangle
(|u|^{p-1}u)\big\|_{L_t^{{q}'}L_x^{{r}'}}.
\end{align*}
Choose $\theta=1-\frac{(p-1)(n-2)}{4}$, then we have by Leibniz rule, H\"older's inequality and Sobolev inequality
\begin{align*}
\big\|\Phi(u)\big\|_{Y(I)}
\leq&CC_1\|u_0\|_{H^1}+CC_1C_2T^{\theta}\|u\|_{L_t^{q}(I;
L^{\frac{nr}{n-r}})}^{p-1}\|u\|_{L_t^{q}(I;
H_r^{1})}.
\end{align*}

Note $\theta>0$ for $p\in[1+\frac 4n,1+\frac{4}{n-2})$ and $\|u\|_{Y(I)}\leq2CC_1\|u_0\|_{H^1}$ if $u\in B$, we see
that for $u\in B$,
\begin{align*}
\big\|\Phi(u)\big\|_{Y(I)}\leq
CC_1\|u_0\|_{H^1}+CC_1C_2T^\theta(2CC_1\|u_0\|_{H^1})^p.
\end{align*}
Taking $T$ sufficiently small such that
$$C_2T^\theta(2CC_1\|u_0\|_{H^1})^p<\|u_0\|_{H^1},$$
we have $\Phi(u)\in B$ for $u\in B$. On the other hand, by the same argument as before, we have for $u,
v\in B$,
\begin{align*}
d\big(\Phi(u),\Phi(v)\big)
\leq&CT^\theta\big(\|u\|_{Y(I)}^{p-1}+\|v\|_{Y(I)}^{p-1}\big)d(u,v).
\end{align*}
Thus we derive by taking $T$ small enough
\begin{equation*}
d\big(\Phi(u),\Phi(v)\big)\leq\tfrac{1}{2}d(u,v).
\end{equation*}

Next we consider the case $a\in(-\frac{4p}{(p+1)^2}\la_n,0)$. In this case we can choose $(q,r)\in \Lambda_0$ such that
\begin{equation*}
\left(\tfrac1q,\tfrac1r\right)=\left(\tfrac1q,\tfrac1p\left(\tfrac1{\widetilde{r}'}+\tfrac{p-1}n\right)\right),
\end{equation*}
where we denote $\widetilde{r}'$ to be
\begin{equation}\label{ryipdy}
\widetilde{r}'=\begin{cases}\left(\tfrac{2n}{n+2}\right)_+,\quad \text{if}\quad
\min\{1-\lambda_n,0\}\leq a<0;\\
(r_1')_+,\quad \text{if} \quad -\frac{4p}{(p+1)^2}\lambda_n<a<\min\{1-\lambda_n,0\}.
\end{cases}
\end{equation}

 Since $a<0$, one has
$1/{r_1'}=({n+\sqrt{(n-2)^2+4a}})/(2n)$. And so we can verify $r\in (r_0,r_1)$ by $a>-{4p\la_n}/{(p+1)^2}$. Hence we
have by Proposition \ref{equiv}
\begin{align*}
\big\|\Phi(u)\big\|_{Y(I)}\leq&C_1\big\|\langle
P_a^{1/2}\rangle\Phi(u)\big\|_{L_t^\infty(I; L_x^2)\cap
L_t^{q}(I;L_x^{r})}.
\end{align*}

Since the operator $P_a^{\frac12}$ commutates with the Schr\"odinger
propagator $e^{itP_a}$, we have by \eqref{instr} and \eqref{Riesz}
\begin{align*}
\big\|\Phi(u)\big\|_{Y(I)}
 \leq& CC_1\|u_0\|_{H^1}+CC_1\big\|\langle
P_a^{1/2}\rangle (|u|^{p-1}u)
\big\|_{L_t^{\widetilde{q}'}L_x^{\widetilde{r}'}(I\times\R^n)},
\end{align*}
where $(\widetilde{q},\widetilde{r})\in \Lambda_0$. Note that
$\widetilde{r}'\in(r_1',r_0')\cap(\frac{n}{n-1},n)\cap (\frac{2n}{n+2},2)$ when $a<0$. By
Proposition \ref{pequiv} and Leibniz rule for the operator $\langle
\nabla\rangle$, we obtain that for $\theta$ as above
\begin{align*}
\big\|\Phi(u)\big\|_{Y(I)}
 \leq& CC_1\|u_0\|_{H^1}+CC_1\big\|\langle
\nabla\rangle (|u|^{p-1}u)
\big\|_{L_t^{\widetilde{q}'}L_x^{\widetilde{r}'}(I\times\R^n)}\\
\leq&
CC_1\|u_0\|_{H^1}+CC_1T^\theta\|u\|^{p-1}_{L^q_t(I;L^{\frac{nr}{n-r}}_x(\R^n))}\|u\|_{L^q_t(I;H^1_r(\R^n))}
\\
\leq&
CC_1\|u_0\|_{H^1}+CC_1T^\theta\|u\|^p_{L^q_t(I;H^1_r(\R^n))},
\end{align*}
where $\frac1{\widetilde{q}'}=\theta+\frac p q$ and
$\frac1{\widetilde{r}'}=\frac {(n-r)(p-1)}{nr}+\frac1r$. For $u\in
B$, one must has $\|u\|_{Y(I)}\leq2CC_1\|u_0\|_{H^1}$, hence we have
\begin{align*}
\big\|\Phi(u)\big\|_{Y(I)}\leq
CC_1\|u_0\|_{H^1}+CC_1C_2T^{\theta}(2CC_1\|u_0\|_{H^1})^p.
\end{align*}
We take $T$ sufficiently small such that
$$C_2T^\theta(2CC_1\|u_0\|_{H^1})^p<\|u_0\|_{H^1},$$
hence $\Phi(u)\in B$ for $u\in B$. Similarly, we have
\begin{equation*}
d\big(\Phi(u),\Phi(v)\big)\leq\tfrac{1}{2}d(u,v).
\end{equation*}

The standard fixed point argument gives a unique solution $u$ of
\eqref{Eq3} on $I\times\R^n$ which satisfies the bound
\eqref{small}.\vspace{0.1cm}

\end{proof}

\begin{lemma}[The boundedness of kinetic energy] For $a>-\la_n$,  there exists $c=c(n,a)>0$ such that
\begin{equation}\label{kener}
\|u(t)\|_{\dot H^1}^2<cE(u(t)).
\end{equation}

\end{lemma}

\begin{proof}
We recall the sharp Hardy's inequality that for $n\geq3$
\begin{equation}\label{shardy}
\int_{\R^n}\tfrac{|u(x)|^2}{|x|^2}dx\leq\tfrac4{(n-2)^2}\int_{\R^n}|\nabla u|^2dx.
\end{equation}
Thus,
\begin{align*}
E(u)=&\tfrac12\int|\nabla
u(t)|^2+\tfrac{a}2\int\tfrac{|u(t)|^2}{|x|^2}+\tfrac1{p+1}\int|u(t)|^{p+1}\\
\geq&\tfrac12\min\left\{1,1+\tfrac{4a}{(n-2)^2}\right\}\int|\nabla u(t)|^2.
\end{align*}
Note that $a>-\la_n$, this implies \eqref{kener}.
\end{proof}

By using Proposition \ref{localwell}, mass and energy conservations  and this lemma, we conclude the proof of global well-posed result of
Theorem \ref{thm}.

\vskip 0.2in

\subsection {Scattering theory} Now we use the global
interaction Morawetz estimate \eqref{interac2}
\begin{equation}\label{rMorawetz}
\big\||\nabla|^{\frac{3-n}4}u\big\|_{L_t^4(\R;L_x^4(\R^n))}^2\leq
C\|u_0\|_{L^2}\sup_{t\in \R}\|u(t)\|_{\dot H^{\frac12}},
\end{equation}
to prove the scattering theory part of Theorem \ref{thm}. Since the construction of the wave operator is
standard, we only show the asymptotic completeness.

Let $u$ be a global solution to \eqref{Eq3}. Using \eqref{kener} and
\eqref{rMorawetz}, we have by interpolation
\begin{equation}\label{vnorm} \|u\|_{L_t^{n+1}L_x^\frac{2(n+1)}{n-1}(\R\times\R^n)}\leq C(E,M),\end{equation}
where the constant $C$ depends on the energy $E$ and mass $M$. Let $\eta>0$ be a small constant to be chosen later and split
$\R$ into $L=L(\|u_0\|_{H^1})$ finite subintervals $I_j=[t_j,t_{j+1}]$ such
that
\begin{equation}\label{equ4.16}
\|u\|_{L_t^{n+1}L_x^\frac{2(n+1)}{n-1}(I_j\times\R^n)}\leq\eta.
\end{equation}
We first consider the case $a\geq0$. Define
$$\big\|\langle\nabla\rangle u\big\|_{S^0(I)}:=\sup_{(q,r)\in\Lambda_0:r\in[2,\min\{n_-,(\frac{2n}{n-2})_-\}]}\big\|\langle\nabla\rangle u\big\|_{L_t^qL_x^r(I\times\R^n)}.$$
Using the Strichartz estimate and \eqref{eqnorm} , we obtain
\begin{align}\label{star}
\big\|\langle\nabla\rangle
u\big\|_{S^0(I_j)}\lesssim&\|u(t_j)\|_{H^1}+\big\|\langle\nabla\rangle(|u|^{p-1}u)
\big\|_{L_t^2L_x^\frac{2n}{n+2}(I_j\times\R^n)}.
\end{align}
Let $\epsilon>0$ to be determined later, and
$r_\epsilon=\frac{2n}{n-(4/(2+\epsilon))}$. On the other hand, we
use the Leibniz rule and H\"older's inequality to obtain
\begin{align*}
\big\|\langle\nabla\rangle(|u|^{p-1}u)
\big\|_{L_t^2L_x^\frac{2n}{n+2}}\lesssim&\big\|\langle\nabla\rangle
u\big\|_{L_t^{2+\epsilon}(I_j;L_x^{r_\epsilon})}\|u\|^{p-1}_{L_t^{\frac{2(p-1)(2+\epsilon)}{\epsilon}}L_x^\frac{n(p-1)(2+\epsilon)}{4+\epsilon}}.\end{align*}

When $n\geq4$, we can choose $\epsilon>0$ to be small enough such
that ${2(p-1)(2+\epsilon)}/{\epsilon}>n+1$ and $2\leq
\frac{n(p-1)(2+\epsilon)}{4+\epsilon}<\tfrac{2n}{n-2}$ for all
$p\in(1+\frac4n,1+\frac{4}{n-2})$. Therefore we use interpolation to
obtain
\begin{align*}
\|u\|_{L_t^{\frac{2(p-1)(2+\epsilon)}{\epsilon}}L_x^\frac{n(p-1)(2+\epsilon)}{4+\epsilon}}\leq C\|u\|^\alpha_{L_t^{n+1}L_x^\frac{2(n+1)}{n-1}(I_j\times\R^n)}\|u\|^\beta_{L_t^{\infty}L_x^\frac{2n}{n-2}(I_j\times\R^n)}\|u\|^\gamma_{L_t^{\infty}L_x^2(I_j\times\R^n)},\end{align*}
where $\alpha>0,\beta,\gamma\geq 0$ satisfy $\alpha+\beta+\gamma=1$ and
\begin{align*}
\begin{cases}
\frac{\epsilon}{2(p-1)(2+\epsilon)}&=\frac{\alpha}{n+1}+\frac{\beta}{\infty}+\frac{\gamma}{\infty},\\
\frac{4+\epsilon}{n(p-1)(2+\epsilon)}&=\frac{(n-1)\alpha}{2(n+1)}
+\frac{(n-2)\beta}{2n}+\frac{\gamma}2.
\end{cases}
\end{align*}
It is easy to verify these requirements for
$p\in(1+\frac4n,1+\frac{4}{n-2})$. Since $r_\epsilon\in[2,n_-]$ and
$r_\epsilon<\tfrac{2n}{n-2}$ for $\epsilon>0$ and $n\geq4$, we have
\begin{align*}
\big\|\langle\nabla\rangle(|u|^{p-1}u)
\big\|_{L_t^2L_x^\frac{2n}{n+2}}\lesssim&\big\|\langle\nabla\rangle
u\big\|_{L_t^{2+\epsilon}(I_j;L_x^{r_\epsilon})}\|u\|^{\alpha(p-1)}_{L_t^{n+1}L_x^\frac{2(n+1)}{n-1}(\R\times\R^n)}
\|u\|^{(\beta+\gamma)(p-1)}_{L_t^\infty H^1_x(I_j\times\R^n)}\\\leq&
C\eta^{\alpha(p-1)}\big\|\langle\nabla\rangle
u\big\|_{S^0(I_j)}.\end{align*} Plugging this into \eqref{star} and
noting that $\alpha(p-1)>0$, we can choose $\eta$ to be small enough
such that
\begin{align*}
\big\|\langle\nabla\rangle
u\big\|_{S^0(I_j)}\leq C(E,M,\eta).\end{align*}
Hence we have by the finiteness of $L$
\begin{align}\label{boundnorm}
\big\|\langle\nabla\rangle
u\big\|_{S^0(\R)}\leq C(E,M,\eta,L).\end{align}

When $n=3$, we choose $\epsilon=2_+$, and then $r_\epsilon=3_-$. If
$p\in(\frac73,4]$, then ${2(p-1)(2+\epsilon)}/{\epsilon}>4$ and
$2\leq \frac{3(p-1)(2+\epsilon)}{4+\epsilon}\leq 6$. Arguing as
before, we can estimate $\|\langle\nabla\rangle u\|_{S^0(I_j)}$. If
$p\in (4,5)$, we use interpolation to show that
\begin{align*}
\|u\|_{L_t^{\frac{2(p-1)(2+\epsilon)}{\epsilon}}L_x^\frac{3(p-1)(2+\epsilon)}{4+\epsilon}}\leq C\|u\|^\alpha_{L_t^{4}L_x^4(I_j\times\R^3)}\|u\|^\beta_{L_t^{\infty}L_x^6(I_j\times\R^3)}\|u\|^\gamma_{L_t^{6}L_x^{18}(I_j\times\R^3)},\end{align*}
where  $\alpha>0,\beta,\gamma\geq 0$ satisfy $\alpha+\beta+\gamma=1$ and
\begin{align*}
\begin{cases}
\frac{\epsilon}{2(p-1)(2+\epsilon)}&=\frac{\alpha}{4}+\frac{\beta}{\infty}+\frac{\gamma}{6},\\
\frac{4+\epsilon}{3(p-1)(2+\epsilon)}&=\frac{\alpha}{4}
+\frac{\beta}{6}+\frac{\gamma}{18}.
\end{cases}
\end{align*} It is easy to solve
these equations for $p\in(4,5)$. Since $r_\epsilon\in[2,3_-]$ for
$\epsilon=2_+$, we have
\begin{align*}
\big\|\langle\nabla\rangle(|u|^{p-1}u)
\big\|_{L_t^2L_x^\frac{6}{5}}\lesssim&\big\|\langle\nabla\rangle
u\big\|_{L_t^{2+\epsilon}(I_j;L_x^{r_\epsilon})}\|u\|^{\alpha(p-1)}_{L_t^4L_x^4(\R\times\R^3)}\|u\|^{\beta(p-1)}_{L_t^\infty
H^1_x(I_j\times\R^3)}\|\langle\nabla\rangle u\|^{\gamma(p-1)}_{L_t^{6}L_x^{\frac{18}{7}}(I_j\times\R^3)}\\&\leq C\eta^{\alpha(p-1)}\big\|\langle\nabla\rangle
u\big\|^{1+\gamma(p-1)}_{S^0(I_j)}.\end{align*}
Hence arguing as above we have \eqref{boundnorm} for $n=3$.

We secondly consider the case $\frac4{(p+1)^2}-\la_n<a<0$ and $n\geq4$. Let
$2^*=\frac{2n}{n-2}$. Define
$$\big\|\langle\nabla\rangle u\big\|_{S^0(I)}:=\sup_{(q,r)\in\Lambda_0:r\in[2,(\frac{2n}{n-2})_-]}\big\|\langle\nabla\rangle u\big\|_{L_t^qL_x^r(I\times\R^n)}.$$
For $1-\la_n<a<0$, we have $[(2^*)',2^*]\subset(r_0,r_1)$ and $2^*,
(2^*)'\in (1,n)$. Using the Strichartz estimate and \eqref{RRiesz} ,
we obtain
\begin{align}\nonumber
\big\|\langle\nabla\rangle
u\big\|_{S^0(I_j)}\lesssim&\|u(t_j)\|_{H^1}+\big\|\langle\nabla\rangle(|u|^{p-1}u)
\big\|_{L_t^2L_x^\frac{2n}{n+2}(I_j\times\R^n)}.
\end{align}
If we choose $r_\epsilon$ as before, we have
$r_\epsilon\in[2,2^\ast]\subset (r_0,r_1)$ in this case. Then we can
closely follow the previous argument to obtain \eqref{boundnorm}.
For $\frac4{(p+1)^2}-\la_n\leq a<1-\la_n$,  by using the Strichartz estimate and \eqref{RRiesz}, we instead obtain for some $\alpha>0,\beta\geq1$
\begin{align*}\nonumber
\big\|\langle\nabla\rangle
u\big\|_{S^0(I_j)}\lesssim&\|u(t_j)\|_{H^1}+\big\|\langle\nabla\rangle(|u|^{p-1}u)
\big\|_{L_t^{q_1'}L_x^{(r_1)_-'}(I_j\times\R^n)}\\
\lesssim&\|u(t_j)\|_{H^1}+\|u\|^{\alpha(p-1)}_{L_t^{n+1}L_x^\frac{2(n+1)}{n-1}(I_j\times\R^n)}\big\|\langle\nabla\rangle
u\big\|^\beta_{S^0(I_j)},
\end{align*}
where $(q_1,(r_1)_-)\in\Lambda_0$ and $r_1$ is given in \eqref{r0r1}. Hence we also obtain \eqref{boundnorm}.

Finally, we utilize \eqref{boundnorm} to show asymptotic completeness. We need to prove that there exist unique $u_\pm$ such that
$$\lim_{t\to\pm\infty}\|u(t)-e^{itP_a}u_\pm\|_{H^1_x}=0,\quad P_a=-\Delta+\tfrac{a}{|x|^2}.$$
By time reversal symmetry, it suffices to prove this for positive
times. For $t>0$, we will show that $v(t):=e^{-itP_a}u(t)$ converges
in $H^1_x$ as $t\to+\infty$, and denote $u_+$ to be the limit. In
fact, we obtain by Duhamel's formula
\begin{equation}\label{equ4.21}
v(t)=u_0-i\int_0^te^{-i\tau P_a}(|u|^{p-1}u)(\tau)d\tau.
\end{equation}
Hence, for $0<t_1<t_2$, we have
$$v(t_2)-v(t_1)=-i\int_{t_1}^{t_2}e^{-i\tau P_a}(|u|^{p-1}u)(\tau)d\tau.$$
Arguing as before, we deduce that for some $\alpha>0,\beta\geq1$
\begin{align*}
\|v(t_2)-v(t_1)\|_{H^1(\R^n)}=&\Big\|\int_{t_1}^{t_2}e^{-i\tau P_a}(|u|^{p-1}u)(\tau)d\tau\Big\|_{H^1(\R^n)}\\
\lesssim&\big\|\langle\nabla\rangle(|u|^{p-1}u)
\big\|_{L_t^2L_x^\frac{2n}{n+2}([t_1,t_2]\times\R^n)}\\
\lesssim&\|u\|_{L_t^{n+1}L_x^\frac{2(n+1)}{n-1}([t_1,t_2]\times\R^n)}^{\alpha(p-1)}\big\|\langle\nabla\rangle
u\big\|^\beta_{S^0([t_1,t_2])}
\\
\to&0\quad \text{as}\quad t_1,~t_2\to+\infty.
\end{align*}
As $t$ tends to $+\infty$, the limitation of \eqref{equ4.21} is well
defined. In particular, we find the asymptotic state
$$u_+=u_0-i\int_0^\infty e^{-i\tau P_a}(|u|^{p-1}u)(\tau)d\tau.$$
Therefore, we conclude the proof of Theorem \ref{thm}.




\begin{center}

\end{center}

\end{document}